\documentclass[a4paper]{article}
\usepackage[utf8]{inputenc}
\usepackage[british]{babel}
\usepackage[T1]{fontenc}
\usepackage{titlefoot}

\usepackage[headheight=3ex,body={14cm,22cm}]{geometry}
\usepackage{mathtools, amssymb, amsthm, amsfonts, mathrsfs, xcolor, booktabs, xtab, enumitem, hyperref, authblk}

\hypersetup{
    colorlinks = true, 
    linkcolor = {black},
    citecolor = {red}, 
    urlcolor = {magenta},
    breaklinks = {true}, 
    pdftitle = {Integrality of twisted L-values of elliptic curves},
    pdfauthor = {Hanneke Wiersema and Christian Wuthrich}
}

\DeclareMathOperator{\Gal}{Gal}
\DeclareMathOperator{\GL}{GL}
\DeclareMathOperator{\SL}{SL}
\DeclareMathOperator{\re}{Re}
\DeclareMathOperator{\im}{Im}
\DeclareMathOperator{\Fr}{Fr}
\DeclareMathOperator{\Spec}{Spec}

\newcommand{\FF}{\mathbb{F}}
\newcommand{\QQ}{\mathbb{Q}}
\newcommand{\ZZ}{\mathbb{Z}}
\newcommand{\RR}{\mathbb{R}}
\newcommand{\CC}{\mathbb{C}}
\newcommand{\PP}{\mathbb{P}}

\definecolor{dblue}{rgb}{0,0,0.7}
\newtheoremstyle{mythm}{11pt}{11pt}{\it\color{dblue}}{}{\bf\color{dblue}}{.}{ }{}
\theoremstyle{mythm}
\newtheorem{thm}{Theorem}
\newtheorem{prop}[thm]{Proposition}
\newtheorem{cor}[thm]{Corollary}
\newtheorem{lem}[thm]{Lemma}
\newtheorem*{ms_con}{Stevens's Manin constant conjecture}

\newcommand\sm[4]{\bigl(\begin{smallmatrix} #1 & #2 \\ #3 & #4 \end{smallmatrix}\bigr)}
\newcommand\LL{\mathscr{L}}
\newcommand\La{L^{{\! \mathrm{a}}}}   
\newcommand\LLa{\LL^{{\mathrm{a}}}}

\begin{document}

\title{Integrality of twisted $L$-values of elliptic curves}

\author[1]{Hanneke Wiersema}
\author[2]{Christian Wuthrich}
\affil[1]{University of Cambridge}
\affil[2]{University of Nottingham}

\amssubj{11G40, (11G05, 11F67)}


\maketitle

\begin{abstract}
  Under suitable, fairly weak hypotheses on an elliptic curve $E/\mathbb{Q}$ and a primitive non-trivial Dirichlet character $\chi$, we show that the algebraic $L$-value $\LL(E,\chi)$ at $s=1$ is an algebraic integer.
  For instance, for semistable curves $\LL(E,\chi)$ is integral whenever $E$ admits no isogenies defined over $\QQ$.
 Moreover we give examples illustrating that our hypotheses are necessary for integrality to hold.
\end{abstract}

\tableofcontents

\section{Introduction}
Let $E$ be an elliptic curve defined over~$\QQ$ and let~$\chi$ be a primitive Dirichlet character.
The value of the $L$-function of $E$ twisted by~$\chi$ at $s=1$ can be normalised by periods to obtain an algebraic value~$\LL(E,\chi)$.
We aim to investigate what conditions on $E$ and $\chi$ guarantee the integrality of $\LL(E,\chi)$.

These values appear in many places in the literature.
For the trivial representation the value $\LL(E,\chi)= L(E,1)/\Omega_+(E)$ is the subject of the Birch and Swinnerton-Dyer conjecture.
The equivariant version of this conjecture, as formulated for instance in~\cite{ebsd}, predicts a precise arithmetic interpretation of $\LL(E,\chi)$ for $\chi$ in great generality.
In Iwasawa theory one interpolates these values $p$-adically as the conductor of $\chi$ is of the form $M\cdot p^n$ with $n>0$ varying to get the $p$-adic $L$-function; see~\cite{mtt}.
In a recent paper, which was the starting point for this work~\cite{dew}, Dokchitser, Evans and the first named author study the Artin formalism of such $L$-values and draw surprising conclusions assuming conjectures like the Birch and Swinnerton-Dyer conjecture.

While it has been known to experts since the 1970s that $\LL(E,\chi)$ is an algebraic number, the above conjectures predict that they are very often algebraic integers.
When there is a torsion point on $E$ whose field of definition is an abelian extension of $\QQ$, then the value $\LL(E,\chi)$ might be non-integral for characters corresponding to that extension.
Our aim is to prove the first general integrality statements confirming the expectations, building on earlier partial results on the denominator, like~\cite{stevens89} and~\cite{wuthrich_integral}.
The main focus concerns subtle questions when dealing with primes of additive reduction that divide the conductor of $\chi$.
Resulting from our work here, it is now possible to find all characters for a given $E$ for which $\LL(E,\chi)$ is non-integral and we have included a complete list for curves of small conductor at the end.

Let us now define $\LL(E,\chi)$ with the aim to state our results.
For a given $\chi$, we will write $m$ for the conductor of $\chi$ and $d$ for its order.
Set $\epsilon=\chi(-1) \in \{\pm 1\}$ depending on whether we deal with an even or odd character.
The Gauss sum of $\chi$ is
\[
  G(\chi) = \sum_{a \bmod m} \chi(a) \exp(2\pi i a/m).
\]
Fix a global Néron differential $\omega$ on $E$.
Let $c_{\infty}$ denote the number of connected components of $E(\RR)$.
We use the following definition for the periods of $E$, which is best suited for Artin formalism as in~\cite{dew}:
\[
  \Omega_+(E) = \int_{E(\RR)}\omega = c_{\infty} \cdot \int_{\gamma^+} \omega
  \qquad\text{ and }\qquad
  \Omega_{-}(E) = \int_{\gamma^-} \omega
\]
where we picked a generator $\gamma^+$ of the subgroup of $H_1\bigl(E(\CC),\ZZ\bigr)$ fixed by complex conjugation and a generator $\gamma^-$ for the subgroup where complex conjugation acts by multiplication by $-1$ in such a way that $\Omega_+(E)>0$ and $\Omega_{-}(E)\in i \RR_{>0}$.

We will use the motivic definition of the $L$-function $L(E,\chi,s)$ given in full detail in Section~\ref{lfn_sec}.
It is directly related to the following definition commonly used for modular forms:
If we write the $L$-function $L(E,s)$ of $E$ as the Dirichlet series $\sum_{n\geq 1} a_n \,n^{-s}$, then we define
\[ \La(E,\chi,s) = \sum_{n\geq 1} \frac{\chi(n)\, a_n}{n^{s}}.\]
If $m$ is coprime to the conductor $N$ of $E$, then our choice of normalisations implies that
$L(E,\chi,s) = \La(E,\bar\chi,s)$,
where $\bar\chi$ is the complex conjugate of $\chi$.
For the general case there is an extra correction factor as explained in Section~\ref{lfn_sec}.

The algebraic $L$-value is
\begin{equation}\label{LL_eq}
  \LL(E,\chi) = \frac{L(E,\chi,1)\cdot m }{G(\chi)\cdot \Omega_{\epsilon}(E)} = \epsilon \cdot \frac{L(E,\chi,1)\cdot G(\bar\chi)}{\Omega_{\epsilon}(E)}.
\end{equation}
In Section~\ref{lfn_sec} we will deduce from the theorem of Manin and Drinfeld~\cite{manin, drinfeld} that $\LL(E,\chi)$ is an algebraic number in the field $\QQ(\zeta_d)$ of values of $\chi$.

First a result if one is willing to assume that the curve is semistable:
\begin{thm}\label{semistable_thm}
  Suppose $E/\QQ$ is a semistable $X_0$-optimal elliptic curve.
  Then $\LL(E,\chi)\in \ZZ[\zeta_d]$ for all non-trivial primitive Dirichlet characters $\chi$ of order $d$.
\end{thm}
In particular, $\LL(E,\chi)$ is integral if $E$ is semistable and admits no isogenies defined over~$\QQ$.

To state a more general result, we need to recall the definition of the Manin constant.
Let $f$ be the newform of level $N$, equal to the conductor of $E$, and weight $2$ associated to the isogeny class of $E$.
Also write $\varphi_0\colon X_0(N)\to E$ for a modular parametrisation of $E$ of minimal degree such that the Manin constant $c_0=c_0(E)$ defined by $\varphi_0^*(\omega) = c_0(E)\cdot 2\pi i f(\tau) \, d\tau$ is positive.
It is known that $c_0$ is an integer.
The original conjecture by Manin states that the Manin constant of the $X_0$-optimal curve in the isogeny class of $E$ is~$1$.
See~\cite{manin_constant} for details on the conjecture and an overview of some results.
The conjecture is verified routinely for all curves in Cremona's database~\cite{cremona}.
Yet there are non-optimal curves for which $c_0>1$.
If one uses the Manin constant $c_1= c_1(E)$ analogously defined with respect to the minimal modular parametrisation $\varphi_1\colon X_1(N)\to E$ one has the following related conjecture (see Conjecture~I in~\cite{stevens89}):
\begin{ms_con}\label{ms_con}
  For any elliptic curve $E/\QQ$, the Manin constant $c_1(E)$ is $1$.
\end{ms_con}

Our main result is the following theorem.

\begin{thm}\label{main_thm}
  Let $E$ be an elliptic curve defined over $\QQ$.
  \begin{enumerate}[label=\alph*)] 
    \item Assume the conjecture that $c_1(E)=1$ holds.
    Then $\LL(E,\chi)\in \ZZ[\zeta_d]$ for all non-trivial primitive Dirichlet characters~$\chi$ of order~$d$ whose conductor $m$ is not divisible by a prime of bad reduction for $E$.
    \item Suppose the Manin constant $c_0(E)$ is $1$.
    Then $\LL(E,\chi)\in \ZZ[\zeta_d]$ for all non-trivial primitive Dirichlet characters~$\chi$ of order~$d$ whose conductor $m$ is not divisible by a prime of additive reduction for $E$.
  \end{enumerate}
\end{thm}

Theorem~\ref{semistable_thm} follows from Theorem~\ref{main_thm} because $c_0(E)=1$ is known for the $X_0$-optimal curve by~\cite{ces} if $E$ is semistable.

One can further relax the restriction on the conductor $m$.
We will prove in Theorem~\ref{la_thm} the integrality in more generality but for the corresponding value $\LLa(E,\chi)$ of the $L$-function $\La(E,\chi,s)$ instead.
We achieve this by studying the integrality property of modular symbols, similarly to the questions discussed in~\cite{wuthrich_integral}.
The closely related question of the integrality of the Stickelberger elements is discussed by Stevens in~\cite{stevens}; in Section~\ref{gal_sec} we will refine his methods.

To conclude the same for $\LL(E,\chi)$ one still needs to check that no prime of additive reduction becomes semistable over the number field fixed by the kernel of $\chi$ as explained in Section~\ref{lfn_sec}.
This is also the reason for the limitation to the following corollary.

\begin{cor}\label{main_cor}
  Let $E$ be an elliptic curve defined over $\QQ$.
  Assume the conjecture that $c_1(E)=1$ holds and that $E$ admits a semistable quadratic twist over $\QQ$.
  Then there are only finitely many $\chi$ such that $\LL(E,\chi)$ is non-integral.
\end{cor}

Throughout we identify primitive Dirichlet character via class field theory with characters of the absolute Galois group of $\QQ$.
Let $K$ be the field fixed by the kernel of such a character $\chi$; it is an abelian extension of $\QQ$.
By the Artin formalism, which holds for the motivic $L$-functions but not for $\La(E,\chi,s)$, we have that
$ L(E/K,s) = \prod_{\chi} L(E,\chi,s)$
where the product runs over all characters of the Galois group of $K/\QQ$.
The Birch and Swinnerton-Dyer conjecture predicts now that there is an integer $c\in\ZZ$ such that
\[
  \prod_{\chi} \LL(E,\chi) = \frac{L(E/K,1)\cdot \sqrt{\lvert \Delta_K\rvert}}{\Omega(E/K)} \overset{\mathrm{?}}{=} \frac{c}{\lvert E(K)_{\mathrm{tors}}\rvert ^2}
\]
where $\Delta_K$ is the discriminant of $K$ and $\Omega(E/K)$ is a product of periods $\Omega_{\pm}(E)$.
Of course, the integer $c$ could be zero in which case one of the $\LL(E,\chi)$ is zero and the following does not say much.
However, when $c\neq 0$, the conjecture now predicts that the possible denominator of $\LL(E,\chi)$ should have something to do with torsion points in $E(K)$.
In Corollary~\ref{lachi_tors_cor} and in Theorem~\ref{main_thm}, we prove indeed a strong link between the existence of new torsion points defined over $K$ and non-integral $\LLa(E,\chi)$.
In the opposite direction, the very last example in Section~\ref{ex_sec} shows that one can have a non-integral motivic $L$-value, yet no new torsion points appearing in the corresponding field.
On the arithmetic side, the right hand side of the above conjectured equality, the explanation for integrality comes from the cancellation between $c$ and the order of the torsion subgroup; this was studied for instance in~\cite{lorenzini}.

\subsubsection*{Overview}
The setup of the paper is as follows.
In Section~\ref{geom_sec} we prove some first integrality results using Birch's formula and the geometry of modular symbols.
In Section~\ref{gal_sec} we obtain further integrality results but now using the Galois action on cusps in $X_1(N)$.
In Section~\ref{tors_sec}, we briefly depart from the modular symbols and record some results about acquiring torsion points in abelian extensions of $\QQ$.
We use these and other results from previous sections to prove some results about integrality of modular symbols in Section~\ref{modsym_sec}.
In Section~\ref{mainthm_sec} we prove one of our main integrality results.
In the penultimate section, Section~\ref{lfn_sec}, we compare the motivic definition of the $L$-function to the automorphic definition.
Finally in Section~\ref{ex_sec} we include some detailed examples to demonstrate why the assumptions in our main theorems can not be weakened.
We finish with a table containing all examples of non-integral $L$-values for elliptic curves with conductor below 100.

\subsubsection*{Acknowledgements}
The authors would like to thank Vladimir Dokchitser and Robert Evans.
When this work was carried out the first named author was supported by the Engineering and Physical Sciences Research Council, through the EPSRC Centre for Doctoral Training in Geometry and Number Theory~[EP/L015234/1] (the London School of Geometry and Number Theory) at University College London.

\section{Integrality using the geometry of modular symbols}
\label{geom_sec}

We aim to show that the algebraic $L$-value for the $L$-function $\La(E,\chi,s)$ is integral in two steps.
First, we will write it as a sum involving only elements in the Néron lattice $\Lambda$.
The second step takes care of the possible denominator $2$ by splitting the sum into two equal parts.

Let $E/\QQ$ be an elliptic curve of conductor $N$.
Let $f$ be the newform of weight $2$ associated to the isogeny class of $E$.
For $r\in\QQ$, define
\begin{equation}\label{modsym_eq}
  \lambda(r) = 2\pi i \int_{i\infty}^{r} f(\tau) d\tau = \frac{1}{c_0}\cdot \int_{\gamma(r)} \omega
\end{equation}
where the first integral follows the vertical line in the upper half plane from $i\infty$ to $r\in\QQ$ and $\gamma(r)$ is the image in $E(\CC)$ of this path under $\varphi_0$.
Let $\Lambda$ be the Néron lattice of $E$, i.e., the set of all values of $\int_\gamma \omega$ as $\gamma$ runs through closed loops in $E(\CC)$.
This can have two possible shapes.
If $E(\RR)$ has $c_{\infty} =2$ connected components (which is illustrated in the picture on the left in Figure~\ref{la_fig} below), then $\Lambda = \tfrac{1}{2}\Omega_+(E) \,\ZZ \oplus \Omega_{-}(E)\,\ZZ$.
Instead if $c_{\infty}=1$ as in the picture on the right in Figure~\ref{la_fig}, then $\Lambda$ is spanned by $\Omega_{+}(E)$ and $\tfrac{1}{2}(\Omega_{+}(E) + \Omega_{-}(E))$.
Note the periods in~\cite{wuthrich_integral, sagemath} are differently normalised: there $\Omega_{+}(E)/c_{\infty}$ is used instead of $\Omega_+(E)$.

\begin{figure}[ht]
  \begin{center}
  \includegraphics[height=5cm]{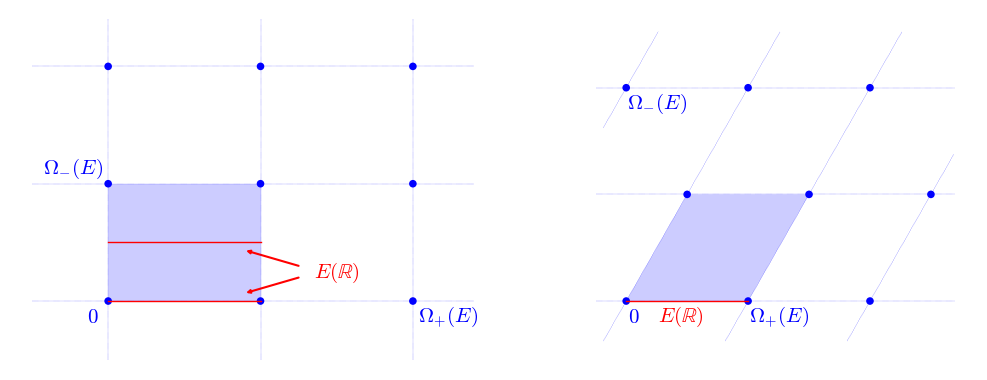}
  \caption{The two types of Néron lattice}\label{la_fig}
  \end{center}
\end{figure}
We note that in both cases we have the following
\begin{equation} \label{reallat_eq}
\bigl\{\re(z)\bigm\vert z\in\Lambda\bigr\} = \tfrac{1}{2}\Omega_{+}(E)\,\ZZ
\end{equation}
and
\begin{equation} \label{imlat_eq}
\bigl\{\im(z)\,i\bigm\vert z\in\Lambda\bigr\} = \tfrac{1}{2}\,c_{\infty}\,\Omega_{-}(E)\,\ZZ\subset \tfrac{1}{2}\Omega_{-}(E)\,\ZZ,
\end{equation}
which we will use frequently to prove our results.

We fix a non-trivial primitive character $\chi$ of conductor $m$ and order $d$.
Set $D = \gcd(m,N)$ and $\delta = \gcd(D,N/D)$.
The primes dividing $D$ are precisely those that are bad for both $E$ and $\chi$, while at those dividing $\delta$ the curve $E$ must have additive reduction.
Write $m = D\cdot \tilde{m}$ and note that $\tilde{m}$ is coprime to $\delta$.

Assume first that $\delta\neq 2$.
For any invertible $a$ modulo $m$, we define $\alpha(a,m)$ to be the least residue of $a\,\tilde{m}$ modulo $\delta$; by definition this means that $-\delta/2 < \alpha(a,m) < \delta/2$ and $\alpha(a,m) \equiv a\,\tilde{m}\pmod{\delta}$.
Note that $a\tilde{m} \not \equiv \delta/2\pmod{\delta}$ unless $\delta=2$ which is why we will treat this case separately.
We define
\[
  \mu\Bigl(\frac{a}{m}\Bigl) = \lambda\Bigl(\frac{a}{m}\Bigr) - \lambda\Bigl(\frac{\alpha(a,m)}{D}\Bigr)
  = 2\pi i \int_{\alpha(a,m)/D}^{a/m} f(\tau)d\tau.
\]
If $\delta=2$, we set simply set $\mu\bigl(\tfrac{a}{m}\bigr)=\lambda\bigl(\tfrac{a}{m}\bigr)$.

We note here that if $m$ is even, which is necessarily the case when $\delta=2$, $m$ must be divisible by 4 since we are interested in non-trivial primitive characters only.

\begin{lem}\label{muint_lem}
  If $\delta\neq 2$, then for all $r$, we have $\mu(r)\in c_0^{-1} \,\Lambda$.
  If $\delta = 2$, then $\mu(r)\in (2c_0)^{-1}\,\Lambda$ for all $r$.
\end{lem}
\begin{proof}
First if $\delta\neq 2$.
By Proposition~2.2 in Manin~\cite{manin} any cusp $\tfrac{a}{m}$ is $\Gamma_0(N)$-equivalent to the cusp $\tfrac{\alpha(a,m)}{D}$.
Since these two cusps are equivalent, the path between them maps to a loop in $X_0(N)(\CC)$.
Therefore its image $\gamma$ in $E(\CC)$ will be closed as well.
Hence
\[
  \mu(r) = 2\pi i \int_{\alpha(a,m)/D}^{a/m} f(\tau)d\tau = \frac{1}{c_0} \int_\gamma \omega\ \in c_0^{-1} \,\Lambda
\]
for all $r\in \QQ$.

The case $\delta=2$ is different.
First both $N$ and $m$ are divisible by $4$.
It follows that the second Hecke operator annihilates the newform $f$.
For the modular symbols $\lambda$, this means that $\lambda\bigl(\tfrac{r}{2}\bigr) + \lambda\bigl(\tfrac{r+1}{2}\bigr)=0$  (see e.g. \cite[(4.2)]{mtt}).
Applied to $r=\tfrac{2a}{m}$, one finds the relation
\begin{equation}\label{t2_eq}
  \lambda\Bigl(\frac{a}{m}\Bigr) = -\lambda\Bigl(\frac{a}{m} + \frac{1}{2}\Bigr)\qquad\text{if $\delta=2$.}
\end{equation}
Since the cusps $\tfrac{a}{m}$ and $\tfrac{a}{m}+\tfrac{1}{2}$ are both $\Gamma_0(N)$-equivalent to $\tfrac{1}{D}$, the difference
\[
  2\,\mu\Bigl(\frac{a}{m}\Bigr) = 2\,\lambda\Bigr(\frac{a}{m}\Bigr) = \lambda\Bigl(\frac{a}{m}\Bigr) - \lambda\Bigl(\frac{a}{m}+\frac{1}{2}\Bigr)
\]
belongs to $c_0^{-1} \,\Lambda$.
\end{proof}

\begin{lem}\label{muneg_lem}
  For all $r\in\QQ$, we have $\mu(-r) = \overline{\mu(r)}$.
\end{lem}
\begin{proof}

  The equality $\lambda(-r) = \overline{\lambda(r)}$ can be verified through explicit computation using the action $\tau\mapsto -\bar\tau$ on the upper half plane and that the modular form $f$ has real coefficients.
  This proves already the case $\delta=2$.

  If $\delta \neq 2$, our choice of representative $\alpha(a,m)$ modulo $\delta$ implies that $\alpha(-a,m) = - \alpha(a,m)$.
  We obtain
  \begin{align*}
    \overline{\mu\Bigl(\frac{a}{m}\Bigr)} &= \overline{\lambda\Bigl(\frac{a}{m}\Bigr)} - \overline{\lambda\Bigl(\frac{\alpha(a,m)}{D}\Bigr)}
    = \lambda\Bigl(-\frac{a}{m}\Bigr) - \lambda\Bigl(-\frac{\alpha(a,m)}{D}\Bigr)\\
    &=  \lambda\Bigl(\frac{-a}{m}\Bigr) - \lambda\Bigl(\frac{\alpha(-a,m)}{D}\Bigr) = \mu\Bigl(\frac{-a}{m}\Bigr).\qedhere
  \end{align*}
\end{proof}

Write $L(E,s)= \sum_{n\geq 1} a_n \, n^{-s}$ for the Dirichlet series for the $L$-function of $E$, which converges absolutely for $\re(s) > \tfrac{3}{2}$.
We define $\La(E,\chi,s)$ as the analytic continuation of the Dirichlet series
\[
  \La(E,\chi,s)=\sum_{n\geq 1} \frac{a_n\, \chi(n)}{ n^{s} }.
\]
This is the $L$-function of the modular form $f$ twisted by $\chi$ as in~\cite{mtt}.
In Section~\ref{lfn_sec}, we will compare $\La(E,\chi,s)$ and $L(E,\chi,s)$.
We decorate the first with an ``a'' for ``automorphic'', while the latter is a motivic $L$-function.
\begin{lem}\label{musum_lem}
  Suppose $\delta\neq m$.
  Then
  \[
    \La(E,\bar\chi,1) =\frac{G(\bar\chi)}{m}\, \sum_{a} \chi(a) \,\mu\bigl(\tfrac{a}{m}\bigr)
  \]
  where the sum runs over all invertible $a$ modulo $m$.
\end{lem}
Note that the condition $m\neq \delta$ is satisfied as soon as $\chi$ is non-trivial and $(m,N)=1$ or, more generally, if no additive place ramifies in $K_{\chi}/\QQ$.
\begin{proof}
  We use Birch's formula (see formula~(8.6) in~\cite{mtt}):
  \begin{equation}\label{birch_eq}
    \La(E,\bar\chi,1) = \frac{G(\bar\chi)}{m} \sum_{a}\, \chi(a)\, \lambda\bigl(\tfrac{a}{m}\bigr)
  \end{equation}
  where the sum runs over $a\in\ZZ/m\ZZ$.
  Again, this proves already the case $\delta=2$.

  Suppose now that $\delta\neq 2$.
  The sum can be rewritten as
  \[
    \sum_{a} \chi(a) \lambda\bigl(\tfrac{a}{m}\bigr) = \sum_{a} \chi(a) \,\mu\bigl(\tfrac{a}{m} \bigr) + \sum_{a} \chi(a)\, \lambda\bigl(\tfrac{\alpha(a,m)}{D}\bigr)
  \]
  and we are left to show that the second summand on the right is equal to zero.
  This sum is equal to
  \begin{equation}\label{dsum_eq}
    \sum_{-\delta/2< x < \delta/2 } \lambda\bigl(\tfrac{x}{D}\bigr) \sum_{\substack{a \bmod m\\ a\tilde{m}\equiv x \pmod{\delta}}} \chi(a).
  \end{equation}
  For a fixed $x$, we wish to show that the last sum on the right is zero.
  As $\tilde{m}$ and $a$ are coprime to $\delta$ it is possible to pick an invertible $y$ modulo $m$ such that $a\tilde{m} \equiv y \pmod{\delta}$.
  Then every $a$ modulo $m$ such that $a\tilde{m} \equiv y\pmod{\delta}$ can be written uniquely as $a = y(1+k\delta)$ for one $0\leq k < m/\delta$.
  Therefore
  \[
    \sum_{\substack{a \bmod m\\ a\tilde{m}\equiv x \pmod{\delta}}} \chi(a) = \sum_{k=0}^{\frac{m}{\delta}-1} \chi\bigl(y(1+k\delta)\bigr)
    = \chi(y)\cdot \sum_{h \in H} \chi(h)
  \]
  where $H$ is the kernel of $(\ZZ/m\ZZ)^\times \to (\ZZ/\delta\ZZ)^\times$.
  Since $m\neq \delta$,  the kernel $H$ is non-trivial as it is impossible that $\delta$ is odd and $m = 2 \delta$ since $m$ is the conductor of a character.
  Now the above sum is $\sum_{h\in H} \chi\vert_{H}(h)$ and by character theory this is $0$ unless $\chi$ restricts to the trivial character on $H$.
  But the latter is impossible as $\chi$ is assumed to be primitive modulo $m$.
\end{proof}

We define
\[
  \LLa(E,\chi) = \frac{\La(E,\bar\chi,1)\cdot m }{G(\bar\chi)\cdot \Omega_{\epsilon}(E)}
\]
analogous to the definition in~\eqref{LL_eq}.

\begin{prop}\label{mneqdelta_prop}
  Assume that the Manin constant $c_0$ for $E$ is $1$ and suppose $m^2\nmid N$.
  Then $\LLa(E,\chi)\in \ZZ[\zeta_d]$.
\end{prop}
\begin{proof}
  Note first that the assumption $m^2\nmid N$ is equivalent to $m\neq \delta$.
  By the definition of $\LLa(E,\chi)$ and Lemma~\ref{musum_lem}, we have
  \[
    \LLa(E,\chi) = \frac{1}{\Omega_\epsilon(E)} \sum_{a\bmod m} \chi(a)\, \mu\bigl(\tfrac{a}{m}\bigr).
  \]
  If $m$ is even then $\chi(m/2)=0$ since $m$ cannot be equal to $2$.
  Therefore for all $m$ we may split the above sum into two sums as
  \begin{align*}
    \LLa(E,\chi) & = \frac{1}{\Omega_\epsilon(E)} \sum_{a=1}^{[\frac{m-1}{2}]} \chi(a)\, \mu\bigl(\tfrac{a}{m}\bigr) + \frac{1}{\Omega_\epsilon(E)} \sum_{a=1}^{[\frac{m-1}{2}]} \chi(-a)\, \mu\bigl(\tfrac{-a}{m}\bigr)\\
     & = \frac{1}{\Omega_\epsilon(E)} \sum_{a=1}^{[\frac{m-1}{2}]} \Bigl( \chi(a)\, \mu\bigl(\tfrac{a}{m}\bigr) + \epsilon\cdot \chi(a) \,\mu\bigl(-\tfrac{a}{m}\bigr)\Bigr).
  \end{align*}
  Using Lemma~\ref{muneg_lem}, one obtains
  \begin{equation}\label{m2_eq}
    \LLa(E,\chi) = \frac{1}{\Omega_\epsilon(E)} \sum_{a=1}^{[\frac{m-1}{2}]} \chi(a) \cdot \Bigl( \mu\bigl(\tfrac{a}{m}\bigr) + \epsilon \cdot \overline{\mu\bigl(\tfrac{a}{m}\bigr)}\Bigr).
  \end{equation}

  Assume now first that $\delta\neq 2$.
  If $\epsilon = +1$, then $\mu(r) + \epsilon \,\overline{\mu(r)}= 2\cdot \re\bigl(\mu(r)\bigr)$.
  By Lemma~\ref{muint_lem}, $\mu(r)\in \Lambda$.
  In either case, whether $\Lambda$ is rectangular or not, the set of $\re(z)$ for $z\in\Lambda$ is $\tfrac{1}{2}\Omega_{+}(E)\,\ZZ$.
  Therefore, if $\epsilon = +1$
  \[
    \frac{\mu\bigl(\tfrac{a}{m}\bigr) + \epsilon \,\overline{\mu\bigl(\tfrac{a}{m}\bigr)}}{\Omega_\epsilon(E)}
  \]
  belongs to $\ZZ$.
  If $\epsilon = -1$ the same argument also works since $\mu(r) - \overline{\mu(r)}= 2\cdot \im \bigl(\mu(r)\bigr)\, i \in \Omega_{-}(E)\,\ZZ$ for both forms of the lattice.
  Since $\chi$ takes values in $\ZZ[\zeta_d]$ when $\chi$ has order $d$, this proves that $\LLa(E,\chi)\in\ZZ[\zeta_d]$.

  The case $\delta=2$, requires more as $\mu(r)=\lambda(r)$ does not necessarily belong to $\Lambda$, but only to $\tfrac{1}{2}\Lambda$ as seen in Lemma~\ref{muint_lem}.
  We now split the sum in equation~\eqref{m2_eq} once more, using the fact that $m$ is divisible by $4$ in this case.
  \begin{multline*}
     \LLa(E,\chi) = \frac{1}{\Omega_\epsilon(E)} \sum_{a=1}^{\frac{m}{4}-1} \chi(a) \cdot \biggl( \lambda\Bigl(\frac{a}{m}\Bigr) +\epsilon \cdot\overline{\lambda\Bigl(\frac{a}{m}\Bigr)}\biggr) +  \\
     + \frac{1}{\Omega_\epsilon(E)} \sum_{a=1}^{\frac{m}{4}-1} \chi(\tfrac{m}{2}-a) \cdot \biggl( \lambda\Bigl(\frac{1}{2} -\frac{a}{m}\Bigr) + \epsilon \cdot\overline{ \lambda\Bigl(\frac{1}{2} - \frac{a}{m} \Bigr) }\biggr)
  \end{multline*}
  We concentrate on the second sum.
  First $\chi\bigl(\tfrac{m}{2} - a\bigr) = \chi(-1) \,\chi(a)\,\chi\bigl(1+\tfrac{m}{2}\bigr)$ (recall that $a$ must be odd and therefore $\tfrac{am}{2}$ is congruent to $\tfrac{m}{2}$ modulo $m$).
  Since $1+\tfrac{m}{2}$ is of order two and because $\chi$ has conductor $m$, we must have $\chi\bigl(1+\tfrac{m}{2}\bigr)= - 1$.
  Further we use~\eqref{t2_eq} and reach
  \begin{align*}
     \LLa(E,\chi) &= \frac{1}{\Omega_\epsilon(E)} \sum_{a=1}^{\frac{m}{4}-1} \chi(a) \cdot
        \biggl( \lambda\Bigl(\frac{a}{m}\Bigr) + \epsilon\cdot \overline{\lambda \Bigl(\frac{a}{m}\Bigr)} -\epsilon \cdot \biggl( \overline{-\lambda\Bigl(\frac{a}{m}\Bigr)} - \epsilon\cdot \lambda\Bigl(\frac{a}{m}\Bigr) \biggr)\biggr)\\
        & = \sum_{a=1}^{\frac{m}{4}-1} \chi(a) \cdot 2 \cdot \frac{\lambda\bigl(\tfrac{a}{m}\bigr) + \epsilon\cdot \overline{\lambda\bigl(\tfrac{a}{m}\bigr)}}{\Omega_{\epsilon}(E)}.
  \end{align*}
  With the extra factor of $2$ and knowing that $\lambda\bigl(\tfrac{a}{m}\bigr)\in\tfrac{1}{2}\Lambda$, we can conclude again.
\end{proof}

For curves with $c_0>1$, we can use the modular parametrisation by $X_1(N)$ instead.
The result will be a bit weaker but it should apply to all curves.
Recall that the Manin constant $c_1$ satisfies $\varphi_1^*(\omega) = c_1\cdot 2\pi i f(\tau) d\tau$, where $\varphi_1:X_1(N)\to E$ is the modular parametrisation of minimal degree, and that it is conjectured to be $1$.
\begin{prop}\label{x1_prop}
  Assume that $c_1=1$.
  Then $\LLa(E,\chi)\in\ZZ[\zeta_d]$ for all non-trivial primitive characters $\chi$ of conductor $m\nmid N$.
\end{prop}
\begin{proof}
  We define an analogue of $\alpha(a,m)$ in this proof to account for our change in parametrisation.
  We let $\beta(a,m)$ be the least residue of $a$ modulo $D$, where $D=\gcd(m,N)$ as before.
  For any $a$ and $b$ that are coprime to $m$, the cusps $\tfrac{a}{m}$ and $\tfrac{b}{m}$ are $\Gamma_1(N)$-equivalent if and only if $a\equiv b\pmod{D}$; see for instance Proposition~3.8.3 in~\cite{diamond_shurman}.
  Assume first that $D\neq 2$ which assures that the least residue $\beta(a,m)$ of $a$ modulo $D$ is well-defined.
  Set $\mu\bigl(\tfrac{a}{m}\bigr) = \lambda\bigl(\tfrac{a}{m}\bigr) - \lambda\bigl(\tfrac{\beta(a,m)}{m}\bigr)$.
  Then it is not hard to check that $\mu\bigl(\tfrac{a}{m}\bigr) \in 1/c_1 \, \Lambda= \Lambda$ and that $\mu(-r) = \overline{\mu(r)}$ for all $r=\tfrac{a}{m}$.
  With these two properties one can now follow precisely the proof of Proposition~\ref{mneqdelta_prop}.
  The corresponding sum that replaces the sum in~\eqref{dsum_eq} is
  \[
    \sum_{\substack{-D/2 < x < D/2\\ (x,m)=1}} \lambda\bigl(\tfrac{x}{m}\bigr) \sum_{\substack{a \bmod m\\ a \equiv x \pmod{D}}} \chi(a)
  \]
  which is $0$ as long as $\chi$ is a primitive character modulo $m$ and $D\neq m$.

  Now to the case when $D=2$.
  Since $m$ is even, it must be divisible by $4$.
  All cusps $\tfrac{a}{m}$ with $a$ coprime to $m$ are $\Gamma_1(N)$-equivalent.
  Write $w = \lambda\bigl(\tfrac{1}{m}\bigr)$.
  Then $w - \overline{w}=\lambda\bigl(\tfrac{1}{m}\bigr) - \lambda\bigl(\tfrac{-1}{m}\bigr)$ is an element of $ 1/c_1 \,\Lambda=\Lambda$.
  More generally $\lambda\bigl(\tfrac{a}{m}\bigr) = w + \nu(a)$ with $\nu(a)\in\Lambda$.
  We compute
  \[
  \LLa(E,\chi) \, \cdot \, \Omega_\epsilon(E) = \sum_{a \bmod m} \chi(a) \cdot\bigl( w + \nu(a) \bigr)
     = \sum_{a=1}^{m/2} \chi(a)\cdot \bigl( \nu(a) + \epsilon\, \overline{\nu(a)}\bigr) + \Bigl( \sum_{a=1}^{m/2}  \chi(a)\Bigr) \cdot \bigl( w + \epsilon \,\overline{w}\bigr).
  \]
  The first sum in the last expression belongs to $\ZZ[\zeta_d]$ because $\nu(a) \in \Lambda$.
  Finally
  \[
    \sum_{a=1}^{m/2}  \chi(a) =\sum_{a=1}^{m/4} (\chi(a) +\chi(\tfrac{m}{2}-a)) = \sum_{a=1}^{m/4} \chi(a) \cdot (1-\epsilon)
  \]
  shows that the second sum in this expression is zero if $\epsilon = 1$.
  Instead if $\epsilon=-1$ then
  \[
  2\, (w - \overline{w} )\cdot  \sum_{a=1}^{m/4} \chi(a)
  \]
  also belongs to $\ZZ[\zeta_d]$ because $w-\overline{w}\in\Lambda$.
\end{proof}

\section{Integrality using the Galois action}\label{gal_sec}

In this section we will use the Galois action on cusps to obtain further cases when $\LLa(E,\chi)$ is integral.
As the statements are a bit stronger, we will use the modular parametrisation $\varphi_1\colon X_1(N) \to E$, but the argument works the same for $X_0(N)$.
For a cusp $r$, we will denote by $P_r= \varphi_1(r)\in E(\overline{\QQ})$.
Under the isomorphism $E(\CC) \cong \CC/\Lambda$ this point corresponds to $c_1\cdot \lambda(r) + \Lambda$.
Throughout the section we will assume that $c_1=1$.
Also, since we know the integrality already when $m\nmid N$ by Proposition \ref{x1_prop}, in this section we will prove results for $m\mid N$.

If $\LLa(E,\chi)$ is not integral, then there is some $a$ for which $\lambda\bigl(\tfrac{a}{m}\bigr)$ does not belong to $\Lambda$.
We will see that this implies that the point $P_{a/m}$ is a non-trivial torsion point.
\begin{lem}\label{pr_lem}
 For any $r=\tfrac{a}{m}$ with $m \mid N$, the torsion point $P_{r}$ is defined over $\QQ(\zeta_{m})$.
\end{lem}
\begin{proof}
  In the proof of Lemma~3.11 in~\cite{stevens89} the action of the Galois group on the cusps in $X_1(N)$ is explicitly given.
  The cusps are defined over $\QQ(\zeta_N)$ and the element $\sigma_b\in\Gal\bigl(\QQ(\zeta_N)/\QQ\bigr)$ sending $\zeta_N$ to $\zeta_N^b$ acts on the cusp represented by $\tfrac{a}{m}$ by sending it to the cusp $\tfrac{a b^*}{m}$ where $bb^*\equiv 1 \pmod{N}$.
  If $b\equiv 1 \pmod{m}$, then $b^*\equiv 1 \pmod{m}$ and hence the cusp $\tfrac{a b^*}{m}$ is $\Gamma_1(N)$-equivalent to $\tfrac{a}{m}$.
  Hence these $\sigma_b$ fix the cusp $\tfrac{a}{m}$ on $X_1(N)$ and hence $P_r$ in $E(\overline{\QQ})$.
\end{proof}
If one uses the parametrisation $\varphi_0\colon X_0(N)\to E$ instead, one can show that the points $P^0_{a/m}= \varphi_0(a/m)$ are defined over $\QQ(\zeta_{\delta})$ using Theorem~1.3.1 in~\cite{stevens}.
In particular $P^0_r\in E(\QQ)$ for all $r$ if $E$ is semistable.

\begin{prop}\label{kchi_prop}
  Assume that $c_1=1$.
  Let $\chi$ be a non-trivial primitive character of conductor $m$ and order $d$ such that $m \mid N$.
  Let $K_\chi$ be the field fixed by the kernel of $\chi$.
  Suppose $K_\chi\not\subset \QQ(P_{1/m})$.
  Then $\LLa(E,\chi) \in \ZZ[\zeta_d]$.
\end{prop}
Before we start with the proof, we introduce the standard notation for the normalised modular symbols $[r]^{\pm}$; which we define by
\[
  \bigl[ r \bigr]^+ = \frac{\re\bigl(\lambda(r)\bigr)}{\Omega_+(E)}\qquad\text{ and }\qquad \bigl[ r \bigr]^- = \frac{\im\bigl(\lambda(r)\bigr)\,i}{\Omega_-(E)} = \frac{\im\bigl(\lambda(r)\bigr)}{\vert \Omega_-(E)\vert}
\]
for any $r\in\QQ$.

From formula~\eqref{birch_eq}, we get that
\[
  \Omega_{\epsilon}(E) \cdot \LLa(E,\chi) = \frac{1}{2} \sum_{a\bmod m} \biggl(\chi(a) \,\lambda\Bigl(\frac{a}{m}\Bigr) + \chi(-a)\, \lambda\Bigl(\frac{-a}{m}\Bigr) \biggr)
  =\sum_{a \bmod m} \chi(a) \,\frac{\lambda\bigl(\tfrac{a}{m}\bigr) + \epsilon \lambda\bigl(\tfrac{-a}{m}\bigr)}{2}
\]
which can be rewritten as
\begin{equation}\label{birch2_eq}
  \LLa(E,\chi) = \sum_{a\bmod m} \chi(a)\cdot \Bigl[\frac{a}{m}\Bigr]^{\epsilon}.
\end{equation}
From $\lambda(-r)=\overline{\lambda(r)}$, it follows that $\bigl[\tfrac{-a}{m}\bigr]^{\epsilon} =\epsilon\cdot \bigl[\tfrac{a}{m}\bigr]^{\epsilon}$.
\begin{proof}
  From the above lemma, we know that $P_{a/m}$ belongs to $\QQ(\zeta_m)$ and how its Galois group $G=\Gal\bigl(\QQ(\zeta_m)/\QQ\bigr)$, identified with $(\ZZ/m\ZZ)^\times$ in the usual way, acts on these points:
  if $b$ is in $(\ZZ/m\ZZ)^\times$ then $\sigma_b(P_{a/m}) = P_{ab^*/m}$ where $b^*$ is the inverse of $b$.
  In particular the Galois group acts transitively on the set of (not necessarily distinct) points $P_{a/m}$ as $a$ varies through invertible elements modulo $m$.

  Let $H$ be the stabiliser of $P_{a/m}$ viewed as a subgroup of $(\ZZ/m\ZZ)^{\times}$ and $F$ the field fixed by $H$, so that $P_{a/m} \in E(F)$.
  Pick a set of coset representatives $U$ of $G/H$.

  Because each $h\in H$ fixes $P_{a/m}$, we find the following relations:
  For each invertible $a$ modulo $m$ and $h\in H$
  \begin{equation}\label{ah_eq}
    \lambda\bigl(\tfrac{ah}{m}\bigr) -\lambda\bigl(\tfrac{a}{m}\bigr) \ \in \frac{1}{c_1}\Lambda=\Lambda.
  \end{equation}
  Let $u \in U$ and for $a\in uH$ define
  \[
    \kappa(a) = \Bigl[\frac{a}{m}\Bigr]^{\epsilon} - \Bigl[\frac{u}{m}\Bigr]^{\epsilon}
  \]
  From the above equation~\eqref{ah_eq}, we see that $\kappa(a)\in\tfrac{1}{2}\ZZ$ and, if $\epsilon=-1$ and $c_{\infty}=2$ then even $\kappa(a)\in\ZZ$.

  Using~\eqref{birch2_eq}, the algebraic $L$-value becomes
  \[
    \LLa(E,\chi) = \sum_{a\bmod m} \chi(a) \biggl(\kappa(a) + \Bigl[\frac{u}{m}\Bigr]^{\epsilon}\biggr)
    =\sum_{a\bmod m} \chi(a) \kappa(a) + \sum_{u \in U} \sum_{h\in H} \chi(uh) \Bigl[\frac{u}{m}\Bigr]^{\epsilon}.
  \]
  The last sum on the right is is equal to
  \[
    \sum_{u\in U}\chi(u)\,\Bigl[\frac{u}{m}\Bigr]^{\epsilon}\cdot \sum_{h\in H}\chi(h).
  \]
  By our hypothesis, $\chi$ is not trivial on $H$ as otherwise $K_\chi\subset F$ and hence this last sum is zero giving $\LLa(E,\chi) = \sum_a \chi(a)\,\kappa(a)$.
  This already proves the lemma in case $\epsilon=-1$ and $c_{\infty}=2$.

  Otherwise, as before, we are left with trying to eliminate the possible denominator~$2$.
  First we assume that  $-1\not\in H$ or equivalently that $P_{1/m}\not\in E(\RR)$.
  Then we may choose $U$ such that, if $u\in U$ then $-u\in U$.
  Then $-a \in -uH$ for all $a\in uH$.
  Therefore
  \[
    \kappa(-a) = \Bigl[\frac{-a}{m}\Bigr]^{\epsilon} - \Bigl[\frac{-u}{m}\Bigr]^{\epsilon} = \epsilon\cdot \kappa(a).
  \]
  We get
  \[
    \LLa(E,\chi) = \sum_{a=1}^{[\frac{m-1}{2}]} \Bigl(\chi(a)\kappa(a) + \chi(-a)\kappa(-a)\Bigr) =
    \sum_{a=1}^{[\frac{m-1}{2}]} \chi(a)\cdot 2 \kappa(a).
  \]
  Since $\kappa(a)\in\tfrac{1}{2}\ZZ$, we can conclude that $\LLa(E,\chi)\in\ZZ[\zeta_d]$.

  We may now assume that $-1$ belongs to $H$ and hence $P_{a/m}\in E(\RR)$ for all $a$.
  This time $-a\in uH$ if $a\in uH$.
  Therefore
  \[
    \kappa(-a) = \epsilon\,\Big[\frac{a}{m}\Bigr]^{\epsilon} - \Bigr[\frac{u}{m}\Bigr]^{\epsilon}
    =\begin{cases} \kappa(a) &\text{ if $\epsilon=+1$}\\ -\kappa(a)- 2 [\frac{u}{m}]^{-}\ &\text{ if $\epsilon=-1$.}\end{cases}
  \]
  Therefore if $\chi$ is even, the same argument as for $-1\not\in H$ works.
  Otherwise, if $\chi$ is odd, then, by the earlier conclusion, we may assume that $c_{\infty}=1$.
  In that case, the lattice is not rectangular and so $P_{u/m}\in E(\RR)$ implies that $\bigl[\frac{u}{m}\bigr]^{-}$ is in $\tfrac{1}{2}\ZZ$ for all $u\in U$.
  Hence in that case $\kappa(-a)$ differs from $-\kappa(a)$ by an integer and we can prove the integrality again.
\end{proof}

This argument uses ingredients similar to those in the result in Theorem~3.14 in~\cite{stevens89}.
In this theorem, Stevens proves an integrality statement for the Stickelberger elements which is a bit weaker than our refined result here.

Note that Proposition~\ref{kchi_prop} implies the following.
\begin{cor}\label{lachi_tors_cor}
If $c_1=1$ and $\LLa(E,\chi)$ is non-integral for some non-trivial $\chi$ of conductor $m$ then $E\bigl(\QQ(\zeta_m)\bigr)$ contains a torsion point that is not defined over $\QQ$.
\end{cor}

\section{Torsion points over abelian extensions}
\label{tors_sec}

We gather some statements about the possibility of acquiring a new torsion point in an abelian extension of~$\QQ$.
In our proofs we make use of Kenku's classification~\cite{kenku} of cyclic isogenies defined over~$\QQ$.

\begin{lem}\label{abtor_odd_lem}
  Let $E/\QQ$ be an elliptic curve and~$p$ an odd prime number and~$n\geq 1$.
  Suppose~$P$ is a point of order~$p^n$ defined over an abelian extension~$K/\QQ$.
  Then~$\QQ(P)$ is contained in a field obtained by adjoining points in the kernel of cyclic isogenies~$\phi\colon E\to E'$ defined over~$\QQ$ whose degrees are powers of~$p$.
\end{lem}
This is a generalisation of Lemma~5 in~\cite{wuthrich_integral}.
\begin{proof}
  Let $G = \Gal\bigl( \QQ(E[p^n])/\QQ\bigr)$ and $N$ its subgroup corresponding to the intermediate field $\QQ(P,\zeta_{p^n})$.
  Since $\QQ(P)$ and $\QQ(\zeta_{p^n})$ are abelian extensions of $\QQ$ so is $\QQ(P,\zeta_{p^n})$.
  Therefore $N$ is a normal subgroup of $G$ with abelian quotient.

  Pick a basis $\{P,Q\}$ of $E[p^n]$ and use it to identify $G$ as a subgroup of $\GL_2(\ZZ/p^n\ZZ)$.
  The subgroup $N$ is then formed by the elements in $G$ of the form $\sm{1}{*}{0}{*}$ that belong to $\SL_2(\ZZ/p^n\ZZ)$.
  Therefore it is a subgroup of matrices of the form $\sm{1}{*}{0}{1}$.
  We have two cases: the special case when $N$ is trivial and the non-trivial case.

  \textit{Case 1:} $N$ is trivial.
  In this case $G$ itself is abelian.
  In order to prove the lemma for $N$ trivial we will explicitly create the field that contains $\QQ(P)$.
  Consider the complex conjugation $g\in G$.
  Since $p$ is odd, there is a basis $\{T_{+}, T_{-}\}$ of $E[p^n]$ such that $g(T_{\pm})=\pm T_{\pm}$.
  For any $h\in G$ we have $h(T_{\pm}) = \pm hg(T_{\pm}) = \pm gh(T_{\pm})$ as $G$ is abelian.
  Therefore $h(T_{\pm})$ is a multiple of $T_{\pm}$ and this shows that all elements in $G$ fix the subgroups generated by $T_{\pm}$.
  The two isogenies $\phi_{\pm}$ whose kernels are generated by $T_{\pm}$ are defined over $\QQ$, and $G$ is contained in the group of diagonal matrices with respect to the new basis $\{T_{+},T_{-}\}$.
  The lemma is then proven in this case as $\QQ(P)\subset \QQ\bigl( E[p^n] \bigr)=\QQ(T_{+},T_{-})$.
  (It turns out that in this Case~1, we are in a very special situation:
  Having two cyclic isogenies of degree $p^n$ leaving $E$, there is a curve in the isogeny class of $E$ over $\QQ$ with an isogeny of degree $p^{2n}$ defined over $\QQ$.
  By Kenku's classification~\cite{kenku}, we know that this only occurs if $p^n$ is $3$ or $5$.)

   \textit{Case 2:} $N$ is non-trivial.
  It is generated by $\sm{1}{p^k}{0}{1}$ for some $0\leq k < n$.
  We consider the action of $G$ and $N$ on the set of cyclic $p^n$-isogenies leaving from $E$, which we may identify with $\PP^1(\ZZ/p^n\ZZ)$ via our chosen basis $\{P,Q\}$ of $E[p^n]$.
  Let $X$ be the set of points fixed by $N$.
  The above generator fixes $(x:y)$ if and only if $y^2\,p^k =0$ in $\ZZ/p^n\ZZ$.
  Therefore $X$ is a set containing $p^m$ elements with $m= [\tfrac{n-k}{2}]$ all of which have the same reduction as $\langle P\rangle =(1:0)$ modulo $p$.
  Since $\# X$ is odd, the complex conjugation $g$ has precisely one fixed point on $X$, say $x_0 = \langle U \rangle$.
  The quotient $G/N$ acts on $X$.
  Since $G/N$ is abelian, we can again conclude that $x_0$ is fixed by all elements in $G/N$.
  Therefore $x_0$ is fixed by all of $G$.

  First, we can treat the easier situation when $x_0=\langle P \rangle$: The isogeny with $P$ in its kernel is then defined over $\QQ$ and the lemma is proven again.
  Since we assume that $N$ is non-trivial, we are in this situation if $n=1$ because then $X$ only contains $\langle P \rangle$.

  Therefore, we are left with the more complicated situation when $n>1$, the group $N$ is non-trivial and $x_0\neq \langle P \rangle$.
  Set $N'$ to be the subgroup of $G$ corresponding to the field $\QQ(P,U,\zeta_{p^n})$.
  Again this is a normal subgroup of $G$ with abelian quotient.
  If $N'$ is trivial, then $G$ is abelian and we can conclude as above.
  Therefore we assume that $N'$ is not trivial and hence it is generated by $h=\sm{1}{p^m}{0}{1}\in N'$ for some $k\leq m < n$.
  For any two $g= \sm{a}{b}{0}{d}$ and $g'=\sm{a'}{b'}{0}{d'}$ in $G$, we must have $gh\in hgN'$.
  This is equivalent to $ab'+bd' \equiv a'b+b'd \pmod{p^m}$.
  This implies that $b'(a-d)\equiv b(a'-d')\pmod{p}$.
  Let $\bar G$ be the image of $G$ in $\GL_2(\FF_p)$, which is the Galois group of $\QQ(E[p])/\QQ$.
  Pick any $\sm{a}{b}{0}{d}\in\bar G$ with $a\neq d$ which exists because the determinant is surjective onto $\FF_p^{\times}$.
  The above congruence implies that the line generated by $\bigl(\begin{smallmatrix} b \\ d-a \end{smallmatrix}\bigr)$ is fixed by $\bar G$.
  In other words besides $p^{n-1}U$ there is a second point $S$ in $E[p]$ such that the subgroup generated by $S$ is fixed by $G$.

  As the isogeny class of $E$ over $\QQ$ now contains a cyclic isogeny of degree $p^{n+1}\geq p^3$, we see that $p=3$ and $n=2$ from Kenku's classification~\cite{kenku}.

  We now change the basis of $E[p^n]=E[p^2]$ by taking $\{U, S'\}$ such that $p S' =S$.
  Then $G$ is in the subgroup of matrices of the form $\sm{a}{b}{0}{d}$ with $p\mid b$.
  The elements in $G$ that fix both $U$ and $S$ are of the form $\sm{1}{b}{0}{d}$ for which $p\mid b$ and $d\equiv 1 \pmod{p}$.
  The point $P$ is of the form $uU + pvS'$ for units $u$ and $v$.
  Therefore all elements that fix $U$ and $S$ also fix $P$.
  We conclude finally that $\QQ(P) \subset \QQ(U,S)$.
  Since both $\langle U\rangle$ and $\langle S\rangle $ are defined over $\QQ$ we have completed the proof in the last case, too.
\end{proof}

\begin{lem}\label{abtor4_lem}
  Let $E/\QQ$ be an elliptic curve.
  Suppose~$P$ is a point of exact order~$4$ defined over an abelian extension~$K/\QQ$.
  Then there is an cyclic isogeny on~$E$ defined over~$\QQ$ of degree~$2$.
\end{lem}
\begin{proof}
  Let $Q=2P$.
  Since $\GL_2(\FF_2)\cong S_3$, there are three possibilities that $Q$ is defined over an abelian extension of $\QQ$.
 For the first two of the possibilities, when $\QQ(Q)$ is either of degree $1$ or of degree $2$, there is already a $2$-isogeny defined over $\QQ$.
 Therefore we assume we are in the third case, that $\QQ(Q)/\QQ$ is a cyclic cubic extension, and show that this contradicts the assumption.

  Consider $G = \Gal\bigl(\QQ(E[4])/\QQ\bigr)$ as a subgroup of $\GL_2\bigl(\ZZ/4\ZZ\bigr)$ with $P$ as the first element of the basis of $E[4]$.
  The image of $G$ in $\GL_2(\FF_2)$ is the subgroup generated by $\sm{0}{1}{1}{1}$.
  Let $H=\Gal\bigl(\QQ(E[4])/\QQ(P,i)\bigr)$ which is a normal subgroup of $G$ contained in the matrices of determinant $1$ and in those that have the form $\sm{1}{*}{0}{*}$.
  Hence $H$ is contained in the cyclic subgroup of order $2$ generated by~$\sm{1}{2}{0}{1}\in\GL_2\bigl(\ZZ/4\ZZ\bigr)$.

  If $H$ is non-trivial then $G$ must belong to the normaliser of $H$ in $\GL_2\bigl(\ZZ/4\ZZ\bigr)$, but that would mean that $G$ is contained in the stabiliser of $P$ which contradicts our assumption on the image of $G$ in~$\GL_2(\FF_2)$.

  We may suppose that $H$ is trivial and hence $\QQ\bigl(E[4]\bigr) = \QQ(P,i)$ is abelian over $\QQ$.
  Since all points in $E[2]$ are defined over $\RR$, the Néron lattice is rectangular.
  Therefore the complex conjugation must be $\sm{1}{0}{0}{-1}\in\GL_2\bigl(\ZZ/4\ZZ\bigr)$ under a suitable, possibly different choice of a basis for $E[4]$.
  All the matrices commuting with that matrix reduce to the identity matrix modulo $2$, which contradicts again our assumption on~$\QQ(Q)$.
\end{proof}

It is tempting to hope that Lemma~\ref{abtor_odd_lem} could be generalised to include $p=2$.
However this turns out to be far from possible.
The possible Galois groups of $\QQ(E[2^\infty])/\QQ$ were determined and listed in~\cite{rouse_zureickbrown}.
We find that of the 1208 possible groups only 582 satisfy the property that the field of definition of all abelian torsion point can be obtained using isogenies over $\QQ$.
In particular this fails for the three groups with the property that $\QQ(E[2])/\QQ$ is cyclic of order $3$.
But there are other more surprising examples: There is a curve with the $2$-primary torsion subgroup defined over the maximal abelian extension of $\QQ$ equal to $\ZZ/16\ZZ\oplus \ZZ/4\ZZ$.

\begin{cor}\label{delta_cor}
  Let $E/\QQ$ be an elliptic curve.
  Suppose ~$P$ is a torsion point of exact order~$t$ defined over an abelian extension~$K/\QQ$.
  Then there exists a cyclic isogeny~$E\to E'$ of degree~$p$ defined over~$\QQ$ for each odd prime divisor~$p$ of~$t$.
  Further if $t$ is even, either
  \begin{itemize}
    \item $4 \nmid t$ and the minimal discriminant~$\Delta$ is a square or
    \item there is a non-trivial cyclic isogeny~$E\to E'$ of degree~$2$ defined over~$\QQ$.
  \end{itemize}
  Moreover if~$t$ is odd then~$\QQ(P)$ is contained in a field obtained by adjoining points in kernels of cyclic isogenies defined over $\QQ$.
\end{cor}
\begin{proof}
  First, $P$ can be written as a linear combination of torsion points $Q_p$ over abelian extensions whose orders are powers of $p\mid t$.
  From the above proof we learn that the only option for $Q_2$ to be defined over an abelian extension without having a rational $2$-isogeny is when $\QQ\bigl(E[2]\bigr)/\QQ$ is a cyclic extension of degree~$3$.
  It is known that this only occurs when the discriminant $\Delta$ is a square, see for instance~5.3a) in~\cite{serregl2}.

  Therefore, if $t$ is odd or $\Delta$ is not a square then the previous two lemmas imply the existence of $p$-isogenies defined over $\QQ$.
  The last sentence is a consequence of Lemma~\ref{abtor_odd_lem} and $\QQ(P)= \QQ\bigl(\{Q_p \bigm\vert p\mid t\}\bigr)$.
\end{proof}

\section{Integrality of modular symbols}\label{modsym_sec}

We record here an auxiliary result that may be of independent interest.
Recall first that modular symbols are the unique rational numbers such that
\[
  \lambda(r) = [r]^+\cdot \Omega_{+}(E) + [r]^{-}\cdot \Omega_{-}(E)
\]
for any $r\in\QQ$.

\begin{prop}\label{modsym_prop}
  Let $E/\QQ$ be an elliptic curve which does not admit any non-trivial isogenies defined over $\QQ$.
  Assume that the Manin constant conjecture $c_0=1$ holds.
  Then~$[r]^\pm$ belongs to~$\tfrac{1}{4}\ZZ$ for any~$r\in\QQ$.
  Furthermore, if the minimal discriminant $\Delta$ is not a square then~$[r]^\pm \in \tfrac{1}{2}\ZZ$.
\end{prop}
\begin{proof}
  Consider the image $P_r$ of the cusp $r\in X_0(N)$ under the modular parametrisation $\varphi_0$.
  Lemma~\ref{pr_lem} implies that $P_r$ is a torsion point defined over an abelian extension.
  From the Lemmas~\ref{abtor_odd_lem} and~\ref{abtor4_lem}, we conclude that $P_r$ has either order $2$ or $P_r=O$.
  Hence the endpoint of the path $\gamma$ in the definition~\eqref{modsym_eq} of $\lambda(r)$ ends at a point in $E[2]$.
  So $\lambda(r)\in \tfrac{1}{2}\Lambda$. Now by ~\eqref{reallat_eq} and ~\eqref{imlat_eq} we obtain~$[r]^\pm\in\tfrac{1}{4}\ZZ$.

  If $\Delta$ is not a square, then Corollary~\ref{delta_cor} shows that $P_r$ has to be $O$ and hence~$\lambda(r)\in\Lambda$.
\end{proof}

We recall here that one can use the modular parametrisation $\varphi_0$ to prove the following, as per our comment after Lemma~\ref{pr_lem}.

\begin{prop}
  Let $E/\QQ$ be a semistable $X_0$-optimal elliptic curve with no non-trivial torsion point defined over~$\QQ$.
  Then~$[r]^\pm \in \tfrac{1}{2}\ZZ$ for all~$r\in\QQ$.
\end{prop}
\begin{proof}
  The Manin constant conjecture $c_0=1$ for $E$ is known in the semistable case by~\cite{ces}.

  By the remark after Lemma~\ref{pr_lem}, all points $P^0_r=\varphi_0(r)$ must be defined over $\QQ$, but since there are no torsion points defined over $\QQ$, we get that~$P^0_r=O$.
  As before~$\lambda(r)\in\Lambda$ and thus~$[r]^\pm \in \tfrac{1}{2}\ZZ$ for all~$r\in\QQ$. 
\end{proof}

It is easy to find examples, even of semistable curves, which have denominator $2$ among the modular symbols:
For instance the curve, labelled 43a1 in Cremona's table~\cite{cremona}, has $\bigl[\tfrac{1}{5}\bigr]^{+} =\bigl[\tfrac{1}{5}\bigr]^{-}=\tfrac{1}{2}$, despite having no isogenies defined over $\QQ$.
Proposition~\ref{modsym_prop} does not rule out that there are examples with denominator $4$.
However it seems very difficult to find any such examples, if they exist at all.

\section{Integrality of the modular \texorpdfstring{$L$}{L}-values}\label{mainthm_sec}
In this short section we prove our main integrality result for $\LLa(E,\chi)$ by combining the results from the previous sections.

\begin{thm}\label{la_thm}
  Let $E/\QQ$ be an elliptic curve of conductor $N$ and $\chi$ a non-trivial character of order $d$ and conductor $m$.
  Assume that $c_1(E)=1$ as conjectured.

  Suppose that $\LLa(E,\chi)\not\in\ZZ[\zeta_d]$.
  Then the minimal discriminant $\Delta$ is a square or there is a cyclic isogeny $E\to E'$ defined over $\QQ$.
  Moreover either
  \begin{itemize}
    \item $c_0>1$ and $m\mid N$ or
    \item $m^2 \mid N$.
  \end{itemize}
  Finally, the field $K_\chi$ fixed by the kernel of $\chi$ is contained in the extension of $\QQ$ obtained by adjoining the points of the kernels of all cyclic isogenies defined over $\QQ$ and all torsion points of order a power of $2$ that are defined over an abelian extension of $\QQ$.
\end{thm}
In the last section, we will give examples explaining why these conditions can not be weakened.

\begin{proof}
  By Proposition~\ref{kchi_prop}, we know that $K_{\chi}$ is contained in the field of definition of $P_{1/m}$, where $P_{1/m}=\varphi_1(\tfrac{1}{m})$.
  This point is defined over an abelian extension by Lemma~\ref{pr_lem}.
  Corollary~\ref{delta_cor} implies that $\Delta$ is a square or there is a cyclic isogeny defined over $\QQ$.
  Proposition~\ref{x1_prop} proves that $m\mid N$.
  Proposition~\ref{mneqdelta_prop} shows that either $m^2\mid N$ or $c_0>1$.
  The statement about $K_\chi$ is a consequence of Lemma~\ref{abtor_odd_lem}.
\end{proof}

\section{Comparison of the \texorpdfstring{$L$}{L}-functions}\label{lfn_sec}
In this section, we compare the motivic definition of the $L$-function to the arithmetic definition.
This allows us to prove our main theorems.

We may view a primitive Dirichlet character $\chi$ as usual as a character on the absolute Galois group of $\QQ$ via setting $\chi(\Fr_p) = \chi(p)$ where $\Fr_p$ is an arithmetic Frobenius element for a prime $p\nmid m$.
For any prime $\ell$, let $\mathcal{F}=\QQ_{\ell}(\zeta_d)$ and let $V_\chi$ be the $1$-dimensional $\mathcal{F}$-vector space on which the absolute Galois group acts by $\chi$.
Let $V_E$ be the dual of $T_{\ell}E \otimes_{\ZZ_{\ell}} \QQ_{\ell}$ where $T_{\ell} E$ is the Tate-module of $E$.

Given such a compatible system of Galois representations $V$, we define its $L$-function as usual by
\[
  L(V,s) = \prod_p \det\Bigl( 1- \Fr_p^{-1} p^{-s} \Bigm\vert V^{I_p} \Bigr)^{-1}
\]
where $I_p$ is the inertia subgroup for $p$ and for each prime $p$ we take $V$ with respect to any prime $\ell\neq p$.
For our definitions of $V_{\chi}$ and $V_E$ above, we obtain $L(\bar\chi,s) = L(V_\chi,s) = \sum_{n\geq 1} \chi(n)\, n^{-s}$ and $L(E,s) = L(V_E,s) = \sum_{n\geq 1} a_n\, n^{-s}$.
We set $L(E,\chi,s) = L(V_E\otimes V_\chi, s)$.

Recall that $\La(E,\chi,s) = \sum_n a_n \chi(n) n^{-s}$, but $\LLa(E,\chi)$ is obtained from $\La(E,\bar\chi,s)$.
The following is not difficult to prove by looking at each local factor in the product.
\begin{lem} \label{la_lemma}
  If $E$ does not have semistable reduction at any prime of $K_\chi$ above a prime $p$ of additive reduction over $\QQ$, then $L(E,\chi,s) = \La(E,\bar\chi,s)$.
\end{lem}

As a consequence of modularity, we obtain that $L(E,\chi,s)$ admits an analytic continuation to all $s\in \CC$.

\begin{proof}[Proof of Theorem~\ref{semistable_thm} and Theorem~\ref{main_thm}]
If no prime of additive reduction divides $m$, then $m^2\nmid N$ and so $c_0=1$ implies that $\LLa(E,\chi)\in\ZZ[\zeta_d]$ by Theorem~\ref{la_thm}.
Part b) of Theorem~\ref{main_thm} then follows from the above lemma as it implies $\LL(E,\chi)= \LLa(E,\chi)$.

If no prime of bad reduction divides $m$ then $m\nmid N$ and so $c_1=1$ implies part a) of Theorem~\ref{main_thm}.
As mentioned before, Theorem~\ref{semistable_thm} is a consequence of Theorem~\ref{main_thm} and the fact that we know $c_0=1$ for the $X_0$-optimal curve as shown in~\cite{ces}.
\end{proof}

The case when the above lemma does not apply is trickier; the two Euler products differ by a finite number of local factors.
We have
\[
  \LL(E,\chi) = \LLa(E,\chi)\cdot \prod_{p \in S} \mathfrak{C}(E,\chi,p)
\]
where $S$ is the set of primes $p$ for which the reduction becomes semistable over $K_\chi$ and the correction factor $\mathfrak{C}(E,\chi,p)$ is the local factor of the Euler product of $L(E,\chi,s)$ at $p$ evaluated at $s=1$.

For instance if $\chi$ is a quadratic character and $E$ achieves good reduction at $p$, then $\mathfrak{C}(E,\chi,p)$ is $p/N_p$ where $N_p$ is the number of points on the reduction of the base-change $E\times \Spec(K_\chi)$ to $K_\chi$ at the unique prime above $p$.
It is therefore not clear that $\LL(E,\chi)$ is integral as these correction factors may introduce new denominators.
In fact, it is not even obvious that the value of $\LL(E,\chi)$ is still in the correct field.
However, Vladimir Dokchitser kindly provided us with the argument to complete this.

\begin{prop}
  For any Dirichlet character $\chi$ of order $d$, we have $\LL(E,\chi)\in\QQ(\zeta_d)$.
\end{prop}
\begin{proof}
The Manin-Drinfeld theorem~\cite{manin, drinfeld} implies that the modular symbols are rational numbers.
Putting this theorem together with Birch's formula~\eqref{birch_eq} we obtain that $\LLa(E,\chi)\in \QQ(\zeta_d)$.

  Let $\chi$ be a Dirichlet character of order $d$ such that the local factor of the Euler product at a prime $p$ for $L(E,\chi,s)$ is non-trivial, while the factor for $L(E,s)$ is trivial.
  So $E$ is an elliptic curve with additive reduction at~$p$, yet does not have additive reduction over $K_\chi$ any more.

Assume first that $E$ has good reduction at primes above $p$ in $K_\chi$.
  Pick a prime $\mathfrak p$ above $p$ in $K_\chi$ and fix $\ell\neq p$.
  The inertia group $I_{\mathfrak p}$ inside the Galois group of $K_\chi/\QQ$ acts on the $\mathcal{F}$-vector space $V_E\otimes \mathcal{F}$ through the character $\chi$ and its inverse $\chi^{-1}$ as the determinant must be trivial by the Weil pairing.
  Therefore the action of $I_{\mathfrak p}$ on $V_E\otimes \mathcal{F}$ will be diagonal for a suitable basis.

  If $\QQ_p^{\mathrm{nr}}$ denotes the maximal unramified extension of $\QQ_p$, then $(K_\chi)_{\mathfrak p}\,\QQ_p^{\mathrm{nr}}/\QQ_p$ is abelian.
  Therefore the action of $\Fr_p$ on $V_E$ commutes with the action by $I_{\mathfrak p}$.
  We conclude that $\Fr_p$ is also diagonal on $V_E\otimes \mathcal{F}$ and hence the characteristic polynomial of $\Fr_p$ has roots in $\mathcal{F}=\QQ_{\ell}(\zeta_d)$ for all $\ell$.
  Since the local factor of the Euler product of $L(E,\chi,s)$ at $p$ is a factor of the one of $L(E/K_\chi,s)$, the correction term $\mathfrak{C}(E,\chi,p)$ is the evaluation of a polynomial over $\QQ(\zeta_d)$ at $p^{-1}$.

  The case when $E$ acquires multiplicative reduction over $K_\chi$ works the same if $V_E$ in the above argument is replaced by its subspace fixed by the inertia subgroup of $K_\chi$ at $p$.
  We note that since $\chi$ is quadratic in this case, the argument before the statement of this proposition also applies.
\end{proof}

\begin{proof}[Proof of Corollary~\ref{main_cor}]
  Let $E^{\psi}$ be the quadratic twist of $E$ that is semistable and denote by $\psi$ the corresponding quadratic character.
  Then $\LL(E,\chi) = \LL(E^{\psi}, \chi\cdot \psi)$ follows from Artin formalism.
  Therefore it is enough to show the corollary for the semistable curve $E^{\psi}$.

  By Theorem~\ref{la_thm}, the set of characters $\chi$ for which $\LLa(E^{\psi},\chi)$ is not integral is finite.
  The result follows after noting that for semistable curves we have $\LLa(E^{\psi},\chi)=\LL(E^{\psi},\chi)$.
\end{proof}

More generally the corollary holds as long as no prime $p$ of additive reduction for $E$ achieves good reduction over a cyclic extension (necessearily of degree $3$, $4$ or $5$) without doing so over a quadratic extension.
The curve of smallest conductor for which this happens is 147b1; two other curves are explained below in Example~5 and~6.
For this reason the tables in Section~\ref{table_sec} are complete with all non-integral values of $\LL(E,\chi)$ since the conductor there is limited to $100$.

\section{Examples}\label{ex_sec}
We wish to end by listing a few examples of non-integral values of $\LL(E,\chi)$ to demonstrate that the assumptions in the statements of the theorems are really needed.

All computations with modular symbols were done using Sage~\cite{sagemath} with the implementation described in~\cite{wuthrich_numerical}.
We used Magma~\cite{magma} for the computation of L-values.
For all examples $c_1=1$ as expected.

\begin{enumerate}[label=\textbf{Example~\arabic*}:, wide, labelindent=0pt]
\item 
For a curve like~$X_0(11)$
\[
 y^2 + y = x^3 - x^2 - 10\,x - 20,
\]
the Manin constant $c_0$ is $1$ and the conductor is square-free.
Therefore all values $\LL(E,\chi)$ will be integral by Theorem~\ref{semistable_thm}.
This, despite the fact that the modular symbols $[r]^+$ have denominator $10$ for many $r$; for instance $\bigl[\tfrac{1}{3}\bigr]^+ = -\tfrac{3}{10}$.
\item 
For the semistable curve~$X_1(11)$
\[
  y^2 + y = x^3 - x^2,
\]
which is 11a3 in Cremona's database, the Manin constant is~$c_0=5$.
The modular symbols like~$[0]^+ = \tfrac{1}{25}$ and~$[\tfrac{1}{2}]^+ = - \tfrac{4}{25}$ have large denominator.
By Proposition~\ref{x1_prop}, only for characters of conductor $11$, we could have non-integral $\LL(E,\chi)$.
Indeed, the character~$\chi$ of conductor~$11$ and order~$5$ sending~$2$ to~$\zeta_5$ produces~$\LL(E,\chi) = \tfrac{1}{5} (2+4\zeta_5+\zeta_5^2+3\zeta_5^3)$.
This value and the conjugate ones under $\Gal\bigl(\QQ(\zeta_5)/\QQ\bigr)$ are all values that are non-integral for this curve.
\item 
Next, we give an illustration of Theorem~\ref{mainthm_sec}. 
In particular, we show that non-integral $L$-values are possible despite having no isogenies defined over $\QQ$.
Let $E$ be the elliptic curve
\[
  E\colon\qquad y^2 = x^3-7\,x + 7,
\]
with label 392f1.
Then $E$ does not admit an isogeny over~$\QQ$.
This curve is not semistable and~$\Delta=2^4\cdot 7^2$ is a square.
For the Dirichlet character~$\chi$ of conductor~$7$ and order~$3$ sending~$3$ to~$\zeta_3$, we find~$\LLa(E,\chi) = \tfrac{1}{2}(2+\zeta_3)$.
Here~$E$ acquires all~$2$-torsion points over the cyclic cubic field with polynomial~$x^3-7x+7$.
The two places of additive reduction are still additive over that field.
Therefore $\LL(E,\chi) = \LLa(E,\chi)$.
\item 
Next, an example with new torsion points of order $5$ over $K_\chi$ demonstrating Corollary~\ref{lachi_tors_cor}:
The curve~75a1 has no torsion points defined over~$\QQ$, but there are $5$-torsion points defined over~$\QQ(\zeta_5)$ which are not fixed by complex conjugation.
For the Dirichlet character~$\chi$ of conductor~$5$ and order~$4$ sending~$2$ to~$i$, we obtain~$\LL(E,\chi)=\tfrac{1}{5}(2-i)$.
The reduction type IV stays the same in the extension $K_\chi/\QQ$ at $p=5$ so that $\LL(E,\chi) = \LLa(E,\chi)$.
\end{enumerate}

Now we pass to some examples where the reduction type changes in the extension $K_\chi/\QQ$, and thus where the conclusion of Lemma~\ref{la_lemma} fails. In particular, we will see the differences between $\LLa(E,\chi)$ and $\LL(E,\chi)$ explicitly.

\begin{enumerate}[resume*]
\item 
Let~$E$ be the curve labelled 162b1.
This curve has a rational $3$-torsion point over~$\QQ$.
This time the curve acquires a new $7$-torsion point over the maximal real subfield~$\QQ(\zeta_9)^+$ of~$\QQ(\zeta_9)$.
For the Dirichlet character of conductor~$9$ and order~$3$ sending~$2$ to~$\zeta_3$, we find~$\LLa(E,\chi) = \tfrac{1}{7}(3+\zeta_3)$.
The correction factor here is the local factor $(1+(1-\zeta_3)T)^{-1}$ evaluated at $\tfrac{1}{3}$ which gives
$\mathfrak{C}(E,\chi,3) = \tfrac{1}{7}(5+\zeta_3)$.
We obtain $\LL(E,\chi) = \tfrac{1}{7}(2+\zeta_3)$, which is still not integral.

This is also an example for which we would not know how to prove that Corollary~\ref{main_cor} holds.
The characters $\psi$ whose $3$-primary part is equal to the above $\chi$ will all have a non-trivial correction term $\mathfrak{C}(E,\psi,3)$.
Although we have verified numerically that many of them have integral $\LL(E,\psi)$, we have no argument to guarantee that they are all integral apart from $\LL(E,\chi)$ and $\LL(E,\bar{\chi})$.
\item 
Next we let $E$ be the curve 150a1.
This curve has additive reduction of type III over $\QQ$, but has good reduction over $\QQ(\zeta_5)$ with $10$ points in the reduction.
The curve has a $2$-torsion point defined over $\QQ$ and over $\QQ(\zeta_5)$ the torsion subgroup is of order $10$.
Take $\chi$ to be the character of order $4$ and conductor $5$ sending $2$ to $i$.
Then the $L$-values are $\LLa(E,\chi) = \tfrac{1}{5}(2+i)$ which has norm $\tfrac{1}{5}$; however $\LL(E,\chi) = \tfrac{1}{10}(3+i)$ of norm $\tfrac{1}{10}$.
The local factor of $\LL(E,\chi,s)$ for the prime $p=5$ is $(1+(2-i)T)^{-1}$ with $T=5^{-s}$.
Again, $\LL(E,\chi)$ is not integral, but this time we even have a new factor $2$ in the denominator.
Also Corollary~\ref{main_cor} is not known to hold here.
\item 
Our final example is the curve 99b1 and the non-trivial character $\chi$ of conductor~$3$.
The values $\LLa(E,\chi) = 2$ is integral, but $\LL(E,\chi)= \tfrac{3}{2}$ is not integral.
The Mordell-Weil group $E(\QQ)$ is cyclic of order $4$.
The curve does not acquire any new torsion points over $K_{\chi} = \QQ(\zeta_3)$.
This may sound surprising when considering the Birch and Swinnerton-Dyer conjecture over both fields.
While the Tate-Shafarevich group is known to be trivial over both fields, the Tamagawa number for $E/\QQ$ at $3$ is $4$, while it is $2$ at the unique place above $3$ in $K_{\chi}$.
Just like the discrepancy between $\LL(E,\chi)$ and $\LLa(E,\chi)$ is due to the change from additive (type I${}_3^*$) reduction to non-split multiplicative (type I${}_6$) reduction, so is the change of the Tamagawa number and the periods as the global minimal equation for $E/\QQ$ is no longer minimal over $K_{\chi}$.

This example illustrates that one cannot hope to link the non-integrality of $\LL(E,\chi)$ to the appearance of new torsion points defined over $K_\chi$ in the same way as we were able to do for $\LLa(E,\chi)$ in Corollary~\ref{lachi_tors_cor}.

\end{enumerate}

\section*{Table}\label{table_sec}


The following table contains all $(E,\chi)$ for which the value of $\LL(E,\chi)$ is not integral and the conductor of $E$ is below $100$.
Only one character $\chi$ for each conjugacy class is listed and the trivial character is omitted for all curves.

The value of $\LLa(E,\chi)$ is only mentioned when it differs from $\LL(E,\chi)$.
The fourth column lists if the minimal discriminant $\Delta$ is a square or not.
We use $t(\QQ)$ and $t(K_{\chi})$ to denote the order of the torsion subgroup of $E(\QQ)$ and $E(K_{\chi})$ respectively.

If the character $\chi$ is quadratic corresponding to $\QQ\bigl(\sqrt{D}\bigr)$, we write $(D/\cdot)$.
Otherwise we give the primitive elements that are sent to the $d$-th root of unity.

\tablecaption{All non-integral $\LL(E,\chi)$ for $N<100$}
\tablehead{
\toprule
Curve &  $c_0$ & $c_\infty$  & $\square?$ &  $t(\mathbb{Q})$ &  $t(K_{\chi})$ &  $m$ & $\chi$ & $\LLa$ & $\LL(E,\chi)$ \\
\midrule}%
\tablefirsthead{
\toprule
Curve &  $c_0$ & $c_\infty$  & $\square?$ &  $t(\mathbb{Q})$ &  $t(K_{\chi})$  &  $m$ & $\chi$ & $\LLa$ & $\LL(E,\chi)$ \\
\midrule}%
\tabletail{\bottomrule}%
\tablelasttail{\bottomrule}%

\begin{xtabular}{lrrlrrrlrr}
 11a3 & 5 & 1 & no & 5 & 25 & 11 & $2 \mapsto \zeta_{5}$ &   & $(2+4\zeta_5+\zeta_5^2+3\zeta_5^3)/5$ \\
 14a4 & 3 & 1 & no & 6 & 18 & 7 & $3 \mapsto \zeta_{3}$ &   & $(1-\zeta_3)/3$ \\
 14a6 & 3 & 2 & no & 6 & 18 & 7 & $3 \mapsto \zeta_{3}$ &   & $(1-\zeta_3)/3$ \\
 15a3 & 2 & 2 & yes & 8 & 16 & 5 & $(5/\cdot)$ &   & $1/2$ \\
 15a7 & 2 & 2 & no & 4 & 8 & 5 & $(5/\cdot)$ &   & $1/2$ \\
 15a8 & 4 & 1 & no & 4 & 8 & 3 & $(-3/\cdot)$ &   & $1/2$ \\
 15a8 & 4 & 1 & no & 4 & 16 & 5 & $2 \mapsto i$ &   & $(1+i)/2$ \\
 15a8 & 4 & 1 & no & 4 & 8 & 5 & $(5/\cdot)$ &   & $1/2$ \\
 20a2 & 2 & 2 & no & 6 & 12 & 5 & $(5/\cdot)$ &   & $1/2$ \\
 20a4 & 2 & 2 & no & 2 & 4 & 5 & $(5/\cdot)$ &   & $3/2$ \\
 21a4 & 2 & 1 & no & 4 & 8 & 3 & $(-3/\cdot)$ &   & $1/2$ \\
 21a4 & 2 & 1 & no & 4 & 8 & 7 & $(-7/\cdot)$ &   & $1/2$ \\
 24a4 & 2 & 1 & no & 4 & 8 & 4 & $(-1/\cdot)$ &   & $1/2$ \\
 24a4 & 2 & 1 & no & 4 & 8 & 3 & $(-3/\cdot)$ &   & $1/2$ \\
 26a3 & 3 & 1 & no & 3 & 9 & 13 & $2 \mapsto \zeta_{3}$ &   & $(2+\zeta_3)/3$ \\
 27a1 & 1 & 1 & no & 3 & 9 & 3 & $(-3/\cdot)$ &   & $1/3$ \\
 27a2 & 1 & 1 & no & 1 & 3 & 3 & $(-3/\cdot)$ &   & $1/3$ \\
 27a3 & 3 & 1 & no & 3 & 9 & 9 & $2 \mapsto \zeta_{3}$ &   & $(2+\zeta_3)/3 $ \\
 27a3 & 3 & 1 & no & 3 & 9 & 3 & $(-3/\cdot)$ &   & $1/3$ \\
 27a4 & 3 & 1 & no & 3 & 9 & 9 & $2 \mapsto \zeta_{3}$ &   & $(2+\zeta_3)/3$ \\
 32a1 & 1 & 1 & no & 4 & 8 & 4 & $(-1/\cdot)$ &   & $1/2$ \\
 32a2 & 2 & 2 & yes & 4 & 8 & 4 & $(-1/\cdot)$ &   & $1/2$ \\
 32a2 & 2 & 2 & yes & 4 & 8 & 8 & $(2/\cdot)$ &   & $1/2$ \\
 32a3 & 2 & 2 & no & 2 & 4 & 4 & $(-1/\cdot)$ &   & $1/2$ \\
 32a4 & 2 & 2 & no & 4 & 8 & 8 & $(2/\cdot)$ &   & $1/2$ \\
 33a2 & 2 & 2 & no & 2 & 4 & 3 & $(-3/\cdot)$ &   & $1/2$ \\
 33a2 & 2 & 2 & no & 2 & 4 & 11 & $(-11/\cdot)$ &   &$1/2$ \\
 35a3 & 3 & 1 & no & 3 & 9 & 7 & $3 \mapsto \zeta_{3}$ &   & $(1-\zeta_3)/3$ \\
 36a1 & 1 & 1 & no & 6 & 12 & 3 & $(-3/\cdot)$ &   & $1/2$ \\
 36a3 & 1 & 1 & no & 2 & 12 & 3 & $(-3/\cdot)$ &   & $1/2$ \\
 40a3 & 2 & 2 & no & 4 & 8 & 5 & $(5/\cdot)$ &   & $1/2$ \\
 45a1 & 1 & 1 & no & 2 & 8 & 3 & $(-3/\cdot)$ &  $1/4$ & $3/16$ \\
 45a2 & 1 & 2 & yes & 4 & 8 & 3 & $(-3/\cdot)$ & $1/2$ & $3/8$ \\
 45a3 & 1 & 2 & no & 2 & 4 & 3 & $(-3/\cdot)$ & $1/2$ & $3/8$ \\
 45a4 & 1 & 2 & yes & 4 & 8 & 3 & $(-3/\cdot)$ &  $ 1 $  & $3/4$ \\
 45a5 & 1 & 2 & yes & 4 & 4 & 3 & $(-3/\cdot)$ &  $ 2 $  & $3/2$ \\
 45a6 & 1 & 1 & no & 2 & 8 & 3 & $(-3/\cdot)$ &  $ 1 $  &  $3/4$ \\
 45a8 & 1 & 1 & no & 2 & 2 & 3 & $(-3/\cdot)$ &  $ 2 $  & $3/2$ \\
 48a1 & 1 & 2 & yes & 4 & 8 & 4 & $(-1/\cdot)$ &   & $1/2$ \\
 48a2 & 1 & 2 & no & 2 & 4 & 4 & $(-1/\cdot)$ &   & $1/2$ \\
 48a4 & 2 & 1 & no & 2 & 8 & 4 & $(-1/\cdot)$ &   & $1/4$ \\
 48a4 & 2 & 1 & no & 2 & 4 & 3 & $(-3/\cdot)$ &   & $1/2$ \\
 49a1 & 1 & 1 & no & 2 & 28 & 7 & $3 \mapsto \zeta_{3} + 1$ &   & $(3+2\zeta_3)/7$\\
 49a1 & 1 & 1 & no & 2 & 4 & 7 & $(-7/\cdot)$ &   & $1/2$ \\
 49a2 & 1 & 2 & no & 2 & 14 & 7 & $3 \mapsto \zeta_{3} + 1$ &   & $(6+4\zeta_3)/7$ \\
 49a3 & 1 & 1 & no & 2 & 4 & 7 & $(-7/\cdot)$ &   & $7/2$ \\
 50a1 & 1 & 1 & no & 3 & 15 & 5 & $2 \mapsto i$ &   & $(1+2i)/5$ \\
 50a2 & 1 & 1 & no & 1 & 5 & 5 & $2 \mapsto i$ &   & $(1+2i)/5$\\
 50b1 & 1 & 1 & no & 5 & 15 & 5 & $(5/\cdot)$ &   & $1/3$ \\
 50b3 & 1 & 1 & no & 1 & 3 & 5 & $(5/\cdot)$ &   & $5/3$\\
 52a2 & 2 & 2 & no & 2 & 4 & 13 & $(13/\cdot)$ &   &$3/2$ \\
 54a3 & 3 & 1 & no & 3 & 9 & 9 & $2 \mapsto \zeta_{3}$ &   & $(1+2\zeta_3)/3$ \\
 54b1 & 1 & 1 & no & 3 & 9 & 3 & $(-3/\cdot)$ &   & $1/3$ \\
 54b2 & 1 & 1 & no & 1 & 3 & 3 & $(-3/\cdot)$ &   & $1/3$ \\
 57b2 & 2 & 2 & no & 2 & 4 & 3 & $(-3/\cdot)$ &   & $1/2$ \\
 57b2 & 2 & 2 & no & 2 & 4 & 19 & $(-19/\cdot)$ &   & $1/2$ \\
 63a1 & 1 & 1 & no & 2 & 8 & 3 & $(-3/\cdot)$ &  $1/4$ & $3/8$ \\
 63a2 & 1 & 2 & yes & 4 & 16 & 3 & $(-3/\cdot)$ &  $1/2$ & $3/4$ \\
 63a3 & 1 & 2 & no & 2 & 8 & 3 & $(-3/\cdot)$ &  $1/2$ & $3/4$ \\
 63a4 & 1 & 2 & yes & 4 & 8 & 3 & $(-3/\cdot)$ &  $ 1 $  & $3/2$ \\
 63a6 & 1 & 1 & no & 2 & 4 & 3 & $(-3/\cdot)$ &  $ 1 $  & $3/2$ \\
 64a1 & 1 & 2 & yes & 4 & 8 & 8 & $(2/\cdot)$ &   & $1/2$ \\
 64a3 & 1 & 2 & no & 4 & 8 & 8 & $(2/\cdot)$ &   & $1/2$ \\
 64a4 & 2 & 1 & no & 2 & 4 & 4 & $(-1/\cdot)$ &   & $1/2$ \\
 64a4 & 2 & 1 & no & 2 & 4 & 8 & $(2/\cdot)$ &   & $1/2$ \\
 64a4 & 2 & 1 & no & 2 & 4 & 8 & $(2/\cdot)$ &   & $1/2$ \\
 72a1 & 1 & 1 & no & 4 & 8 & 3 & $(-3/\cdot)$ &  $1/2$ & $3/8$ \\
 72a1 & 1 & 1 & no & 4 & 4 & 12 & $(3/\cdot)$ &  $ 1 $  & $3/2$ \\
 72a2 & 1 & 2 & yes & 4 & 8 & 3 & $(-3/\cdot)$ &  $ 1 $  & $3/4$ \\
 72a2 & 1 & 2 & yes & 4 & 8 & 12 & $(3/\cdot)$ &  $ 1 $  & $3/2$ \\
 72a3 & 1 & 2 & no & 2 & 4 & 3 & $(-3/\cdot)$ &  $ 1 $  & $3/4$\\
 72a4 & 1 & 2 & yes & 4 & 4 & 3 & $(-3/\cdot)$ &  $ 2 $  &$3/2$ \\
 72a4 & 1 & 2 & yes & 4 & 16 & 12 & $(3/\cdot)$ &  $ 1 $  &$3/2$ \\
 72a5 & 1 & 2 & no & 2 & 8 & 12 & $(3/\cdot)$ &  $ 1 $  & $3/2$ \\
 72a6 & 1 & 1 & no & 2 & 2 & 3 & $(-3/\cdot)$ &  $ 2 $  & $3/2$ \\
 75a1 & 1 & 1 & no & 1 & 5 & 5 & $2 \mapsto i$ &   & $(2-i)/5$ \\
 75b1 & 1 & 1 & no & 2 & 16 & 5 & $2 \mapsto i$ &   & $(3+i)/2$\\
 75b1 & 1 & 1 & no & 2 & 8 & 5 & $(5/\cdot)$ &  $1/4$ & $5/16$ \\
 75b2 & 1 & 2 & yes & 4 & 16 & 5 & $(5/\cdot)$ &  $1/4$ & $ 5/16$\\
 75b3 & 1 & 2 & yes & 4 & 16 & 5 & $(5/\cdot)$ &  $1/2$ & $5/8$ \\
 75b4 & 1 & 2 & no & 2 & 8 & 5 & $(5/\cdot)$ &  $1/4$ & $5/16$ \\
 75b5 & 1 & 2 & yes & 4 & 8 & 5 & $(5/\cdot)$ &  $ 1 $  & $5/4$ \\
 75b6 & 1 & 1 & no & 2 & 8 & 5 & $(5/\cdot)$ &  $ 1 $  & $5/4$\\
 75b7 & 1 & 2 & no & 2 & 4 & 5 & $(5/\cdot)$ &  $ 2 $  & $5/2$\\
 75b8 & 1 & 1 & no & 4 & 4 & 5 & $(5/\cdot)$ &  $ 2 $  & $5/2$\\
 77b3 & 3 & 1 & no & 3 & 9 & 7 & $3 \mapsto \zeta_{3}$ &   & $(2+\zeta_3)/3$\\
 80a2 & 2 & 2 & no & 2 & 4 & 4 & $(-1/\cdot)$ &   & $1/2$ \\
 80b1 & 1 & 1 & no & 2 & 12 & 4 & $(-1/\cdot)$ &   & $1/3$ \\
 80b2 & 2 & 2 & no & 2 & 6 & 4 & $(-1/\cdot)$ &   & $1/3$ \\
 80b2 & 2 & 2 & no & 2 & 4 & 5 & $(5/\cdot)$ &   & $1/2$ \\
 80b4 & 2 & 2 & no & 2 & 4 & 5 & $(5/\cdot)$ &   & $1/2$ \\
 90c1 & 1 & 1 & no & 4 & 12 & 3 & $(-3/\cdot)$ &  $1/3$ & $1/2$ \\
 90c3 & 1 & 1 & no & 12 & 12 & 3 & $(-3/\cdot)$ &  $ 1 $  & $3/2$ \\
 98a1 & 1 & 1 & no & 2 & 36 & 7 & $3 \mapsto \zeta_{3} + 1$ &   & $(4+5\zeta_3)/3$ \\
 98a1 & 1 & 1 & no & 2 & 12 & 7 & $(-7/\cdot)$ &  $ 1/3 $  & $7/18$ \\
 98a2 & 1 & 2 & no & 2 & 18 & 7 & $3 \mapsto \zeta_{3} + 1$ &   & $(8+10\zeta_3)/3$\\
 98a2 & 1 & 2 & no & 2 & 6 & 7 & $(-7/\cdot)$ & $2/3$ & $7/9$ \\
 98a3 & 1 & 1 & no & 2 & 12 & 7 & $(-7/\cdot)$ &  $ 1 $  & $7/6$ \\
 98a4 & 1 & 2 & no & 2 & 6 & 7 & $(-7/\cdot)$ &  $ 2 $  & $7/3$ \\
 98a5 & 1 & 1 & no & 2 & 4 & 7 & $(-7/\cdot)$ &  $ 3 $  & $7/2$ \\
 99b1 & 1 & 2 & no & 4 & 4 & 3 & $(-3/\cdot)$ &  $ 2 $  & $3/2$ \\
 99b2 & 1 & 2 & yes & 4 & 4 & 3 & $(-3/\cdot)$ &  $ 2 $  & $3/2$ \\
 99b3 & 1 & 2 & no & 2 & 4 & 3 & $(-3/\cdot)$ &  $ 2 $  & $3/2$ \\
 99b4 & 1 & 1 & no & 2 & 2 & 3 & $(-3/\cdot)$ &  $ 2 $  & $3/2$ \\
 99d1 & 1 & 1 & no & 1 & 5 & 3 & $(-3/\cdot)$ &  $1/5$ & $3/25$ \\
 99d2 & 1 & 1 & no & 1 & 5 & 3 & $(-3/\cdot)$ &  $ 1 $  & $3/5$ \\
\bottomrule
  \end{xtabular}

\bibliographystyle{amsplain}
\bibliography{ilv}

\end{document}